\documentclass[12pt]{article}
\usepackage{amsthm,amsfonts,amssymb,amsmath}
\IfFileExists{stmaryrd.sty}{\usepackage{stmaryrd}}{}
\usepackage{cite}
\usepackage{epsfig}
\usepackage{url}
\usepackage{xcolor,tikz}
\usetikzlibrary{positioning,arrows.meta,calc,matrix,shapes.geometric,decorations}
\usepackage{lmodern}
\usepackage{multicol,graphicx}
\usepackage{fullpage}
\IfFileExists{bbm.sty}{\usepackage{bbm}}{}
\usepackage{svg}
\usepackage[shortlabels]{enumitem}
\usepackage{float,needspace}
\usepackage[hidelinks]{hyperref}
\newcommand{\graphsmall}{\fontsize{9.5}{11.5}\selectfont}
\definecolor{col1}{HTML}{0072B2}
\definecolor{col2}{HTML}{D55E00}
\definecolor{col3}{HTML}{009E73}
\definecolor{col4}{HTML}{8C568A}
\definecolor{col5}{HTML}{333333}
\tikzset{
  edge/.style={line width=1.3pt,line cap=round},
  e1/.style={edge,draw=col1,solid},
  e2/.style={edge,draw=col2,dash pattern=on 4pt off 2.5pt},
  e3/.style={edge,draw=col3,dash pattern=on 0pt off 3pt},
  e4/.style={edge,draw=col4,dash pattern=on 4pt off 2.5pt on 0pt off 2.5pt},
  e5/.style={edge,draw=col5,
             dash pattern=on 5pt off 2.5pt on 0pt off 2.5pt on 0pt off 2.5pt},
  bridge/.style={preaction={draw=white,solid,line width=5.5pt}},
  vertex/.style={line width=1pt,inner sep=0pt,outer sep=0pt,
                 font=\sffamily\bfseries\fontsize{7}{8.4}\selectfont,text=black},
  v1/.style={vertex,circle,minimum size=6.2mm,draw=col1,fill=col1!18!white},
  v2/.style={vertex,rectangle,minimum size=6.2mm,draw=col2,fill=col2!18!white},
  v3/.style={vertex,regular polygon,regular polygon sides=3,
            minimum size=8.2mm,draw=col3,fill=col3!18!white},
  v4/.style={vertex,diamond,aspect=1,minimum size=8.2mm,
            draw=col4,fill=col4!18!white},
  v5/.style={vertex,regular polygon,regular polygon sides=5,
            minimum size=8.2mm,draw=col5,fill=col5!12!white},
  edgelabel/.style={font=\sffamily\bfseries\graphsmall,fill=white,
                    inner sep=3pt,text=black},
  vlabel/.style={font=\graphsmall,inner sep=1pt,text=black},
  block/.style={font=\fontsize{12}{14.4}\selectfont,text=black!65}
}

\newcommand{\GraphCoordinates}{%
  \coordinate (a1) at (-3.4,4.1);
  \coordinate (b1) at (-3.4,1.5);
  \coordinate (p14) at (-5.1,2.8);
  \coordinate (p13) at (-3.4,2.8);
  \coordinate (p12) at (-1.7,2.8);
  \coordinate (a2) at (3.4,4.1);
  \coordinate (b2) at (3.4,1.5);
  \coordinate (p21) at (1.7,2.8);
  \coordinate (p24) at (3.4,2.8);
  \coordinate (p23) at (5.1,2.8);
  \coordinate (a3) at (3.4,-1.5);
  \coordinate (b3) at (3.4,-4.1);
  \coordinate (p34) at (1.7,-2.8);
  \coordinate (p31) at (3.4,-2.8);
  \coordinate (p32) at (5.1,-2.8);
  \coordinate (a4) at (-3.4,-1.5);
  \coordinate (b4) at (-3.4,-4.1);
  \coordinate (p41) at (-5.1,-2.8);
  \coordinate (p42) at (-3.4,-2.8);
  \coordinate (p43) at (-1.7,-2.8);
}
\newcommand{\GraphLabels}{%
  \foreach \i in {1,2,3,4}{%
    \node[vlabel,above=5mm] at (a\i) {$a_\i$};
    \node[vlabel,below=5mm] at (b\i) {$b_\i$};
  }
  \node[vlabel,above right=3.5mm] at (p12) {$v_{1,2}$};
  \node[vlabel,above left=3.5mm] at (p14) {$v_{1,4}$};
  \node[vlabel,above left=3.5mm] at (p21) {$v_{2,1}$};
  \node[vlabel,above right=3.5mm] at (p23) {$v_{2,3}$};
  \node[vlabel,below right=3.5mm] at (p32) {$v_{3,2}$};
  \node[vlabel,below left=3.5mm] at (p34) {$v_{3,4}$};
  \node[vlabel,below left=3.5mm] at (p41) {$v_{4,1}$};
  \node[vlabel,below right=3.5mm] at (p43) {$v_{4,3}$};
  \node[vlabel,left=5mm] at (p13) {$v_{1,3}$};
  \node[vlabel,right=5mm] at (p24) {$v_{2,4}$};
  \node[vlabel,right=5mm] at (p31) {$v_{3,1}$};
  \node[vlabel,left=5mm] at (p42) {$v_{4,2}$};
  \node[block] at (-5.3,4.4) {$B_1$};
  \node[block] at (5.3,4.4) {$B_2$};
  \node[block] at (5.3,-4.4) {$B_3$};
  \node[block] at (-5.3,-4.4) {$B_4$};
}
\newcommand{\GraphCrossEdges}[6]{%
  \draw[e#1] (p12) -- (p21);
  \draw[e#2] (p14) .. controls (-6.5,2.8) and (-6.5,-2.8) .. (p41);
  \draw[e#3] (p23) .. controls (6.5,2.8) and (6.5,-2.8) .. (p32);
  \draw[e#4] (p34) -- (p43);
  \draw[e#5,bridge] (p13) -- (p31);
  \draw[e#6,bridge] (p24) -- (p42);
}
\newcommand{\GraphFrame}{%
  \pgfresetboundingbox
  \path[use as bounding box] (-6.65,-6.25) rectangle (6.65,4.9);
}

\numberwithin{equation}{section}

\newtheorem{thm}[equation]{Theorem}

\newtheorem{lem}[equation]{Lemma}

\theoremstyle{definition}

\newtheorem*{ai}{AI Declaration}
\newtheorem*{ack}{Acknowledgements}

\theoremstyle{remark}

\title{The List Total Colouring Conjecture is False}
\author{Jonathan A. Noel\thanks{Department of Mathematics and Statistics, University of Victoria, Victoria, B.C., Canada. E-mail: {\tt noelj@uvic.ca}. Research supported by NSERC Discovery Grant RGPIN-2021-02460.}}

\DeclareTextCompositeCommand{\v}{OT1}{l}{l\nobreak\hspace{-.1em}'}
\DeclareTextCompositeCommand{\v}{OT1}{t}{t\nobreak\hspace{-.1em}'\nobreak\hspace{-.15em}}

\begin{document}
\maketitle

\begin{abstract}
We exhibit a cubic graph with total chromatic number equal to four and list total chromatic number five. This disproves the List Total Colouring Conjecture of Borodin--Kostochka--Woodall, Juvan--Mohar--\v{S}krekovski and Hilton--Johnson from the late 1990s. ChatGPT 6 Astra Ultra discovered the counterexample with little input from the author.
\end{abstract}

\section{Introduction}

A \emph{total colouring} of a loopless multigraph $G$ assigns a colour to each vertex and edge so that adjacent vertices receive different colours, edges sharing an endpoint receive different colours, and each edge receives a colour different from those of its endpoints. The \emph{total chromatic number} of $G$, denoted $\chi''(G)$, is the minimum number of colours in any total colouring of $G$. The \emph{list total chromatic number}, denoted $\chi''_\ell(G)$, is the minimum $k$ such that, for every assignment of lists of $k$ colours to the vertices and edges of $G$, there exists a total colouring in which each vertex and edge receives a colour from its list. These notions are closely related to the usual chromatic number, list chromatic number, chromatic index and list chromatic index of a graph $G$, denoted by $\chi(G),\chi_\ell(G),\chi'(G)$ and $\chi_\ell'(G)$, respectively.

The \emph{List Total Colouring Conjecture (LTCC)}, proposed independently by Borodin, Kostochka and Woodall~\cite{BorodinKostochkaWoodall97}, Juvan, Mohar and \v{S}krekovski~\cite{JuvanMoharSkrekovski98} and Hilton and Johnson~\cite{HiltonJohnson99}, says that every multigraph $G$ satisfies 
\[\chi''_\ell(G)=\chi''(G).\] 
This is closely related to the well-studied \emph{List Edge Colouring Conjecture (LECC)}, formulated independently by Vizing, Gupta, Albertson and Collins,
and Bollob\'as and Harris (see~\cite{BorodinKostochkaWoodall97}), that $\chi_\ell'(G)=\chi'(G)$ for every multigraph $G$. A conjecture of Gravier and Maffray~\cite{GravierMaffray97} that is even stronger than the LECC says that every claw-free graph $G$ satisfies $\chi_\ell(G)=\chi(G)$. 

In support of the LTCC, Juvan, Mohar and \v{S}krekovski~\cite{JuvanMoharSkrekovski98} proved that it holds for graphs of maximum degree at most two and that every cubic graph $G$ satisfies $\chi''_\ell(G)\leq 5$. Additional results and discussion of list total colouring can be found in~\cite{Woodall01,KostochkaWoodall02,Woodall10,Woodall07,HetheringtonWoodall06,Woodall06}.

For a graph $G$ and integer $k\geq2$, $G^k$ is the graph obtained from $G$ by adding an edge between any two non-adjacent vertices at distance at most $k$ in $G$. Kostochka and Woodall~\cite{KostochkaWoodall01} proposed the \emph{List Square Colouring Conjecture (LSCC)} that $\chi_\ell(G^2)=\chi(G^2)$ for every graph $G$. The LSCC was motivated in part by its connection to the LTCC; indeed, the LSCC implies the LTCC. The LSCC was disproved by Kim and Park~\cite{KimPark15}. Zhu (see~\cite{KimKwonPark15,KosarPetrickovaReinigerYeager14}) asked whether there exists $k\geq3$ such that every graph $G$ satisfies $\chi_\ell(G^k)=\chi(G^k)$. This, too, was answered negatively for all $k\geq3$ in~\cite{KimKwonPark15,KosarPetrickovaReinigerYeager14}. 

The main result of this paper is that the List Total Colouring Conjecture is false.

\begin{thm}
\label{th:main}
There exists a simple cubic graph $G$ on $20$ vertices such that $\chi''(G)=4$ and $\chi''_\ell(G)=5$.
\end{thm}

We present the counterexample and prove Theorem~\ref{th:main} in the next section. 

\section{The Counterexample}

The graph $G$ is constructed as follows; see also Figure~\ref{fig:total-colouring} for a depiction of $G$ and a total colouring with four colours. Let $B_1,B_2,B_3,B_4$ be four vertex-disjoint copies of $K_{2,3}$. The vertices of $B_i$ are $a_i,b_i$, which we call the \emph{private} vertices of $B_i$, and $v_{i,j}$ for all $j\in\{1,2,3,4\}\setminus\{i\}$, which we call the \emph{terminal} vertices of $B_i$. The set $\{a_i,b_i\}$ is the part of the bipartition of $B_i$ of cardinality two. There is exactly one edge between $B_i$ and $B_j$ for $i\neq j$; it connects $v_{i,j}$ to $v_{j,i}$. We call the edge between $B_i$ and $B_j$ a \emph{cross edge}. It is clear that $G$ is cubic, simple, and has $20$ vertices.  This graph previously appeared in work of Abreu, Goedgebeur,
Labbate and Mazzuoccolo~\cite{AbreuGoedgebeurLabbateMazzuoccolo19}
as an example of a cubic graph admitting no strong Wormald
colouring. It is graph~27005 in House of
Graphs~\cite{CoolsaetDhondtGoedgebeur23}.\footnote{
\url{https://houseofgraphs.org/graphs/27005}}

Since $\Delta(G)=3$, we have $\chi''(G)\geq4$. Let us show that $\chi''(G)=4$ by exhibiting a total colouring with $4$ colours. Colour all of the private vertices and the cross edges with colour $4$. Note that no two private vertices are adjacent, no private vertex is incident to a cross edge, and no two cross edges share an endpoint. Thus, colouring these elements with colour $4$ does not create a conflict. No additional elements receive colour $4$. 

Next, we colour the terminal vertices according to the following table:
\[
\begin{array}{c|cccc}
c(v_{i,j}) & j=1 & j=2 & j=3 & j=4 \\ \hline
i=1 & - & 2 & 3 & 1 \\
i=2 & 1 & - & 3 & 2 \\
i=3 & 1 & 2 & - & 3 \\
i=4 & 2 & 3 & 1 & -
\end{array}
\]
Note that the entry in the $i$th row and $j$th column of the table differs from that of the $j$th row and $i$th column for all $i\neq j$. Therefore, adjacent terminal vertices receive different colours. Thus, since the private vertices get colour $4$, the colouring is proper on the vertices. Finally, for each terminal vertex $v$ of $B_i$ that is coloured with some colour $c(v)\in \{1,2,3\}$, we colour the edge $va_i$ with the colour in $\{1,2,3\}$ that follows $c(v)$ and we colour $vb_i$ with the colour in $\{1,2,3\}$ that precedes $c(v)$, viewing the elements of $\{1,2,3\}$ cyclically. By construction, no vertex has the same colour as an edge that is incident to it, and no two incident edges have the same colour either. Therefore, $\chi''(G)=4$.  

\begin{figure}[htbp]
\centering
\resizebox{0.80\linewidth}{!}{%
\begin{tikzpicture}[x=1cm,y=1cm]
  \GraphCoordinates

  % Terminal colours:
  % c(v_{i,j}) = j for distinct i,j in {1,2,3};
  % c(v_{i,4}) = i for i in {1,2,3};
  % c(v_{4,1}), c(v_{4,2}), c(v_{4,3}) = 2,3,1.
  %
  % At a terminal of colour r, its internal edges to a_i and b_i
  % have colours sigma(r) and sigma^2(r), where sigma=(1 2 3).

  % Internal edges of B_1.
  \draw[e3] (a1) -- (p12);
  \draw[e1] (a1) -- (p13);
  \draw[e2] (a1) -- (p14);
  \draw[e1] (b1) -- (p12);
  \draw[e2] (b1) -- (p13);
  \draw[e3] (b1) -- (p14);

  % Internal edges of B_2.
  \draw[e2] (a2) -- (p21);
  \draw[e1] (a2) -- (p23);
  \draw[e3] (a2) -- (p24);
  \draw[e3] (b2) -- (p21);
  \draw[e2] (b2) -- (p23);
  \draw[e1] (b2) -- (p24);

  % Internal edges of B_3.
  \draw[e2] (a3) -- (p31);
  \draw[e3] (a3) -- (p32);
  \draw[e1] (a3) -- (p34);
  \draw[e3] (b3) -- (p31);
  \draw[e1] (b3) -- (p32);
  \draw[e2] (b3) -- (p34);

  % Internal edges of B_4.
  \draw[e3] (a4) -- (p41);
  \draw[e1] (a4) -- (p42);
  \draw[e2] (a4) -- (p43);
  \draw[e1] (b4) -- (p41);
  \draw[e2] (b4) -- (p42);
  \draw[e3] (b4) -- (p43);

  % Shared cross-edge paths; white underlays mark crossings.
  \GraphCrossEdges{4}{4}{4}{4}{4}{4}

  % Vertices are drawn last, covering the ends of their incident edges.
  % The digit inside each symbol is its colour.
  \node[v4] at (a1) {4};
  \node[v4] at (b1) {4};
  \node[v2] at (p12) {2};
  \node[v3] at (p13) {3};
  \node[v1] at (p14) {1};

  \node[v4] at (a2) {4};
  \node[v4] at (b2) {4};
  \node[v1] at (p21) {1};
  \node[v3] at (p23) {3};
  \node[v2] at (p24) {2};

  \node[v4] at (a3) {4};
  \node[v4] at (b3) {4};
  \node[v1] at (p31) {1};
  \node[v2] at (p32) {2};
  \node[v3] at (p34) {3};

  \node[v4] at (a4) {4};
  \node[v4] at (b4) {4};
  \node[v2] at (p41) {2};
  \node[v3] at (p42) {3};
  \node[v1] at (p43) {1};

  \GraphLabels

  % The legend remains informative when printed in greyscale.
  \begin{scope}[yshift=-5.3cm]
    \foreach \c/\x/\desc in {1/-4.8/solid,2/-1.6/dashed,3/1.6/dotted,4/4.8/dash-dot}{
      \node[v\c] at (\x-0.75,0) {\c};
      \draw[e\c] (\x-0.15,0) -- (\x+0.75,0);
      \node[font=\graphsmall] at (\x,-0.72) {\c: \desc};
    }
  \end{scope}
  \GraphFrame
\end{tikzpicture}%
}
\caption{The graph $G$ with a total $4$-colouring. Colours are indicated by the vertex shapes and edge patterns, as described in the legend.}
\label{fig:total-colouring}
\end{figure}
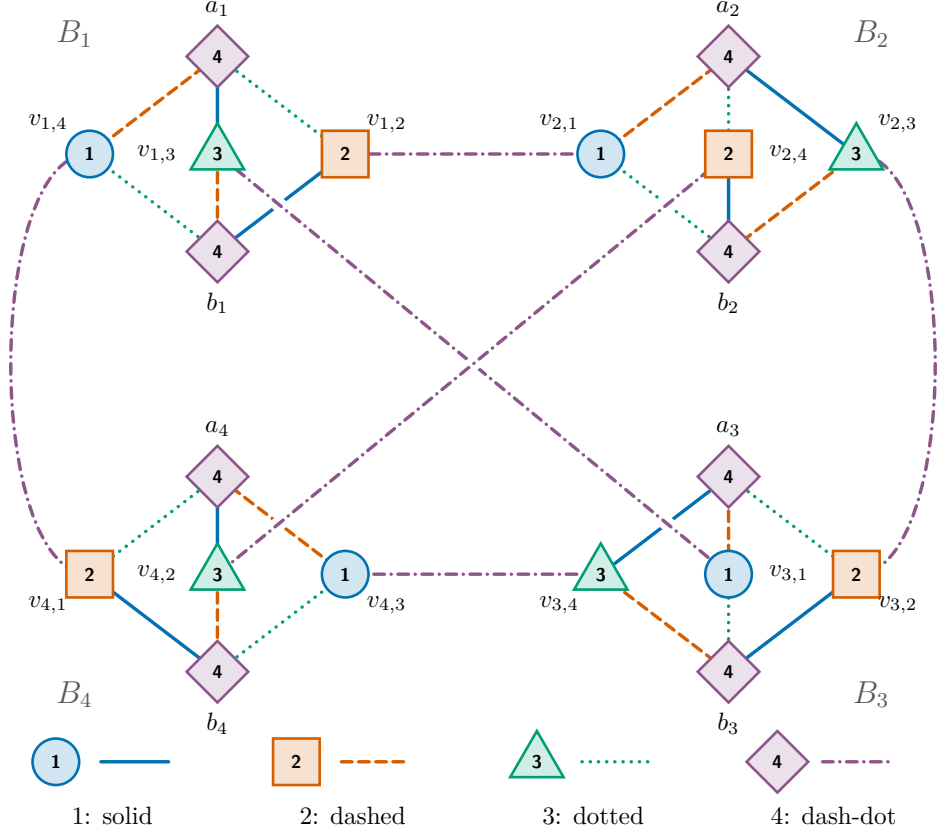

By~\cite[Theorem~3.1]{JuvanMoharSkrekovski98}, $\chi''_\ell(G)\leq 5$. So, to complete the proof of Theorem~\ref{th:main}, we need to exhibit an assignment $L$ of lists of cardinality $4$ to the vertices and edges of $G$ such that $G$ cannot be total coloured from these lists. 

First, for $i\in\{1,2,3,4\}$, assign the list $\{1,2,3,4,5\}\setminus\{i\}$ to every vertex and edge of $B_i$. The cross edges between $B_1$ and $B_3$ and between $B_2$ and $B_4$ have list $\{1,2,3,4,5\}\setminus \{5\}$. Finally, the cross edge from $B_i$ to $B_{i+1}$ receives the list $\{1,2,3,4,5\}\setminus\{i+2\}$, viewing the indices $i\in\{1,2,3,4\}$ cyclically. See Figure~\ref{fig:list-assignment} for a depiction of the list assignment $L$. 

\begin{figure}[htbp]
\centering
\resizebox{0.80\linewidth}{!}{%
\begin{tikzpicture}[x=1cm,y=1cm]
  \GraphCoordinates

  % All 24 internal edges: style i means the list omits colour i.
  \draw[e1] (a1) -- (p12);
  \draw[e1] (a1) -- (p13);
  \draw[e1] (a1) -- (p14);
  \draw[e1] (b1) -- (p12);
  \draw[e1] (b1) -- (p13);
  \draw[e1] (b1) -- (p14);
  \draw[e2] (a2) -- (p21);
  \draw[e2] (a2) -- (p23);
  \draw[e2] (a2) -- (p24);
  \draw[e2] (b2) -- (p21);
  \draw[e2] (b2) -- (p23);
  \draw[e2] (b2) -- (p24);
  \draw[e3] (a3) -- (p31);
  \draw[e3] (a3) -- (p32);
  \draw[e3] (a3) -- (p34);
  \draw[e3] (b3) -- (p31);
  \draw[e3] (b3) -- (p32);
  \draw[e3] (b3) -- (p34);
  \draw[e4] (a4) -- (p41);
  \draw[e4] (a4) -- (p42);
  \draw[e4] (a4) -- (p43);
  \draw[e4] (b4) -- (p41);
  \draw[e4] (b4) -- (p42);
  \draw[e4] (b4) -- (p43);

  % Shared cross-edge paths; white underlays mark crossings.
  \GraphCrossEdges{3}{2}{4}{1}{5}{5}

  % Vertices are drawn last, covering the ends of their incident edges.
  % The digit inside each symbol is its OMITTED colour.
  \node[v1] at (a1) {1};
  \node[v1] at (b1) {1};
  \node[v1] at (p12) {1};
  \node[v1] at (p13) {1};
  \node[v1] at (p14) {1};
  \node[v2] at (a2) {2};
  \node[v2] at (b2) {2};
  \node[v2] at (p21) {2};
  \node[v2] at (p23) {2};
  \node[v2] at (p24) {2};
  \node[v3] at (a3) {3};
  \node[v3] at (b3) {3};
  \node[v3] at (p31) {3};
  \node[v3] at (p32) {3};
  \node[v3] at (p34) {3};
  \node[v4] at (a4) {4};
  \node[v4] at (b4) {4};
  \node[v4] at (p41) {4};
  \node[v4] at (p42) {4};
  \node[v4] at (p43) {4};

  \GraphLabels

  % Omitted colours on all six joining edges. These unframed digits
  % are edge labels, not additional vertices. The side-curve midpoints
  % are (+/-6.15,0), obtained from their cubic Bezier control points.
  \node[edgelabel] at (0,2.8) {3};
  \node[edgelabel] at (0,-2.8) {1};
  \node[edgelabel] at (-6.15,0) {2};
  \node[edgelabel] at (6.15,0) {4};
  \path (p13) -- node[pos=.30,edgelabel] {5} (p31);
  \path (p24) -- node[pos=.30,edgelabel] {5} (p42);

  % Key: a numbered shape and an edge pattern, with their allowed list.
  \begin{scope}[yshift=-5.3cm]
    \foreach \c/\x/\lst in {
      1/-5.2/{2,3,4,5},2/-2.6/{1,3,4,5},3/0/{1,2,4,5},
      4/2.6/{1,2,3,5},5/5.2/{1,2,3,4}}{
      \node[v\c] at (\x-0.65,0) {\c};
      \draw[e\c] (\x-0.10,0) -- (\x+0.80,0);
      \node[font=\graphsmall] at (\x,-0.72) {$\{\lst\}$};
    }
  \end{scope}
  \GraphFrame
\end{tikzpicture}%
}
\caption{An assignment of lists of size four to the vertices and edges of $G$ for which $G$ admits no total colouring. Each vertex and edge has list $\{1,2,3,4,5\}\setminus\{i\}$, where $i$ is indicated by its colour, shape or edge pattern.}
\label{fig:list-assignment}
\end{figure}
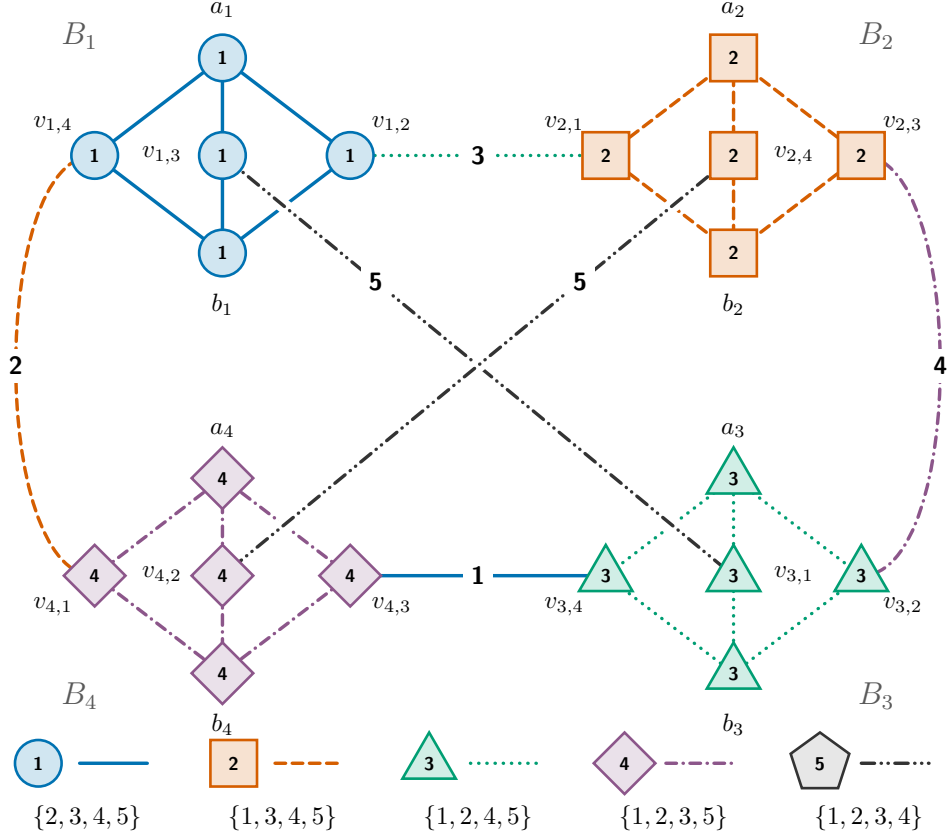

Let us show that $G$ cannot be total coloured from $L$. Suppose, to the contrary, that such a colouring $c$ exists. First, we establish the following lemma regarding $K_{2,3}$.

\begin{lem}
\label{lem:sameColour}
In every total colouring of $K_{2,3}$ with $4$ colours, the two vertices in the part of the bipartition of cardinality $2$ receive the same colour. 
\end{lem}

\begin{proof}
Consider a total colouring with $4$ colours, say red, blue, green and yellow. Let the vertices of $K_{2,3}$ be $a,b,v_1,v_2,v_3$ where $\{a,b\}$ is the part of the bipartition of size $2$. 

Suppose, to the contrary, that $a$ is red and $b$ is blue. Then each of the vertices $v_1,v_2,v_3$ must be green or yellow. By the pigeonhole principle, we can assume that $v_1$ and $v_2$ are both green. Since $a$ is incident to three edges, each of the colours blue, green and yellow must appear on one such edge. However, since $v_1$ and $v_2$ are green, the only edge incident to $a$ that can be green is $av_3$. Similarly, the only edge incident to $b$ that can be green is $bv_3$. However, the vertex $v_3$ is now incident to two green edges, which is a contradiction. 
\end{proof}

Now, since all elements of $B_i$ have the same list, Lemma~\ref{lem:sameColour} implies that $c(a_i)=c(b_i)$ for all $i\in\{1,2,3,4\}$. Let $s_i:=c(a_i)=c(b_i)$ and note that $s_i\neq i$ because $i$ is not in the lists of $a_i$ and $b_i$. Define $S_i:=\{i,s_i\}$ for $i\in\{1,2,3,4\}$.

Let us show that, for any distinct pair of indices $i,j\in\{1,2,3,4\}$, the colour of the cross edge $v_{i,j}v_{j,i}$ must be in $S_i$. First, the colour of this edge must differ from the colour of the vertex $v_{i,j}$ and the edges $v_{i,j}a_i$ and $v_{i,j}b_i$. Note that these three colours are distinct, none equals $i$ by definition of the list assignment, and none equals $s_i$ because $a_i$ and $b_i$ are coloured with $s_i$ and each of these three elements is adjacent or incident to at least one of $a_i$ and $b_i$. Therefore, the colour of the cross edge $v_{i,j}v_{j,i}$ is contained in $\{1,2,3,4,5\}\setminus\{c(v_{i,j}),c(v_{i,j}a_i),c(v_{i,j}b_i)\}=S_i$. By symmetry of $i$ and $j$, the colour of $v_{i,j}v_{j,i}$ is contained in $S_j$ as well, and so $S_i\cap S_j\neq\emptyset$. 

Next, we use the fact that $S_i\cap S_j\neq\emptyset$ and $i\in S_i$ for all $i,j\in\{1,2,3,4\}$ to show that there is a colour $t\in \bigcap_{i=1}^4S_i$. First, since $i\in S_i$ and $|S_i|=2$ for all $i\in\{1,2,3,4\}$, these four sets cannot all be equal. So, let $i,j\in\{1,2,3,4\}$ such that $S_i\neq S_j$. Then we can write $S_i=\{t,x\}$ and $S_j=\{t,y\}$ for some $t\in\{1,2,3,4,5\}$. We claim that the other two sets, say $S_\ell$ and $S_k$ for $\{\ell,k\}=\{1,2,3,4\}\setminus\{i,j\}$, must also contain $t$. If $t\notin S_\ell$, then we must have $S_\ell=\{x,y\}$ in order for $S_\ell$ to have non-empty intersection with $S_i$ and $S_j$. Since $i\in S_i$, $j\in S_j$ and $\ell\in S_\ell$, this implies that $S_i,S_j$ and $S_\ell$ are the three $2$-element subsets of $\{i,j,\ell\}$. However, since $k\in S_k$ and $k$ is distinct from $i,j,\ell$, we get that $S_k$ cannot have non-empty intersection with all three of $S_i,S_j,S_\ell$, a contradiction. 

So, let $t\in \bigcap_{i=1}^4S_i$. To complete the proof we show that, for every possible choice of $t$, there is a cross edge that cannot be coloured from its list. Suppose, for example, that $t=1$. Then $S_3=\{1,3\}$ and $S_4=\{1,4\}$. As proven above, the colour of the edge $v_{3,4}v_{4,3}$ must be in $S_3\cap S_4= \{1\}$. However, this is a contradiction because the list of this edge omits colour $1$. Similarly, for any $t\in\{1,2,3,4\}$, the edge $v_{t-1,t-2}v_{t-2,t-1}$ cannot be coloured, where the elements of $\{1,2,3,4\}$ are viewed cyclically. Also, if $t=5$, then the edge $v_{1,3}v_{3,1}$ cannot be coloured. This completes the proof.

Given that $\chi_\ell''(G)$ is not equal to $\chi''(G)$ in general, an interesting direction for future work is to expand the gap between these two quantities. Let us observe that a gap of at least $3$ would disprove the List Edge Colouring Conjecture. By an argument in, e.g.~\cite[pp.~191--192]{Tuza97}, it holds that $\chi''_\ell(G)\leq \chi'_\ell(G)+2$ for every graph $G$. Also, clearly, $\chi''(G)\geq \chi'(G)$. Thus, if $G$ is a graph such that $\chi''_\ell(G)\geq \chi''(G)+3$, then we would have
\[
\chi'_\ell(G)
\geq \chi''_\ell(G)-2
\geq \chi''(G)+1
> \chi'(G),
\]
thereby disproving the LECC. Thus, if we believe that the LECC is true, then we should expect $\chi_\ell''(G)$ to be fairly close to $\chi''(G)$ in general. Perhaps $\chi_\ell''(G)\leq \chi''(G)+1$ holds for all multigraphs $G$?

\begin{ack}
The author thanks Wouter Cames van Batenburg for pointing out
the earlier appearance of the counterexample graph, and
Cames van Batenburg, Jan Goedgebeur and Giuseppe Mazzuoccolo
for helpful discussions concerning extensions of the
construction and its possible connection with strong
Wormald colourings.
\end{ack}

\begin{ai}
On September 24, 2026, the author prompted ChatGPT 6 Astra Ultra to disprove the List Total Colouring Conjecture and it produced the counterexample in this paper. The author checked the arguments and rewrote the exposition based on drafts generated by ChatGPT. During rewriting, ChatGPT provided assistance with proofreading, suggesting references and answering questions about the arguments. The figures were also generated by ChatGPT. The author takes full responsibility for correctness. 
\end{ai}


\begin{thebibliography}{99}

\bibitem{AbreuGoedgebeurLabbateMazzuoccolo19}
M.~Abreu, J.~Goedgebeur, D.~Labbate, and G.~Mazzuoccolo.
Colourings of cubic graphs inducing isomorphic monochromatic
subgraphs.
\emph{J. Graph Theory}, 92(4):415--444, 2019.

\bibitem{BorodinKostochkaWoodall97}
O.~V. Borodin, A.~V. Kostochka, and D.~R. Woodall.
List edge and list total colourings of multigraphs.
\emph{J. Combin. Theory Ser. B}, 71(2):184--204, 1997.

\bibitem{CoolsaetDhondtGoedgebeur23}
K.~Coolsaet, S.~D'hondt, and J.~Goedgebeur.
House of Graphs 2.0: A database of interesting graphs and more.
\emph{Discrete Appl. Math.}, 325:97--107, 2023.

\bibitem{GravierMaffray97}
S.~Gravier and F.~Maffray.
Choice number of $3$-colorable elementary graphs.
\emph{Discrete Math.}, 165/166:353--358, 1997.
Graphs and combinatorics (Marseille, 1995).

\bibitem{HetheringtonWoodall06}
T.~J. Hetherington and D.~R. Woodall.
Edge and total choosability of near-outerplanar graphs.
\emph{Electron. J. Combin.}, 13(1):Research Paper 98, 7~pp., 2006.

\bibitem{HiltonJohnson99}
A.~J.~W. Hilton and P.~D. Johnson, Jr.
The Hall number, the Hall index, and the total Hall number of a graph.
\emph{Discrete Appl. Math.}, 94(1--3):227--245, 1999.

\bibitem{JuvanMoharSkrekovski98}
M.~Juvan, B.~Mohar, and R.~{\v{S}}krekovski.
List total colourings of graphs.
\emph{Combin. Probab. Comput.}, 7(2):181--188, 1998.

\bibitem{KimKwonPark15}
S.-J. Kim, Y.~S. Kwon, and B.~Park.
Chromatic-choosability of the power of graphs.
\emph{Discrete Appl. Math.}, 180:120--125, 2015.

\bibitem{KimPark15}
S.-J. Kim and B.~Park.
Counterexamples to the list square coloring conjecture.
\emph{J. Graph Theory}, 78(4):239--247, 2015.

\bibitem{KosarPetrickovaReinigerYeager14}
N.~Kosar, S.~Petrickova, B.~Reiniger, and E.~Yeager.
A note on list-coloring powers of graphs.
\emph{Discrete Math.}, 332:10--14, 2014.

\bibitem{KostochkaWoodall01}
A.~V. Kostochka and D.~R. Woodall.
Choosability conjectures and multicircuits.
\emph{Discrete Math.}, 240(1--3):123--143, 2001.

\bibitem{KostochkaWoodall02}
A.~V. Kostochka and D.~R. Woodall.
Total choosability of multicircuits. I, II.
\emph{J. Graph Theory}, 40(1):26--43, 44--67, 2002.

\bibitem{Tuza97}
Zs.~Tuza.
Graph colorings with local constraints---a survey.
\emph{Discuss. Math. Graph Theory}, 17(2):161--228, 1997.

\bibitem{Woodall01}
D.~R. Woodall.
List colourings of graphs.
In J.~W.~P. Hirschfeld, editor,
\emph{Surveys in Combinatorics, 2001},
volume~288 of \emph{London Mathematical Society Lecture Note Series},
pages 269--301.
Cambridge University Press, 2001.

\bibitem{Woodall06}
D.~R. Woodall.
Total 4-choosability of series-parallel graphs.
\emph{Electron. J. Combin.}, 13(1):Research Paper 97, 36~pp., 2006.

\bibitem{Woodall07}
D.~R. Woodall.
Some totally 4-choosable multigraphs.
\emph{Discuss. Math. Graph Theory}, 27(3):425--455, 2007.

\bibitem{Woodall10}
D.~R. Woodall.
The average degree of a multigraph critical with respect to
edge or total choosability.
\emph{Discrete Math.}, 310(6--7):1167--1171, 2010.

\end{thebibliography}
\end{document}